\documentclass[12pt,a4paper]{amsart}

\usepackage{euscript,amsfonts,amssymb,amsmath,amscd,amsthm,enumerate,hyperref,mathrsfs}

\usepackage{tikz,bm,float}
\usetikzlibrary{calc}
\colorlet{lgray}{white!85!black}
\colorlet{lred}{white!75!red}

\usepackage{comment}

\usepackage{graphicx}

\usepackage{color}

\usepackage[margin=1.1in]{geometry}
\newtheorem{theorem}{Theorem} 
\newtheorem*{theorem*}{Theorem}
\newtheorem{lemma}[theorem]{Lemma}

\newtheorem{proposition}[theorem]{Proposition}

\newtheorem{corollary}[theorem]{Corollary}

\theoremstyle{remark}
\newtheorem{remark}[theorem]{Remark}

\numberwithin{equation}{section} \numberwithin{theorem}{section}

\newcommand{\N}{\mathbb N}

\newcommand{\Z}{\mathbb Z}

\newcommand{\R}{\mathbb R}

\newcommand{\C}{\mathbb C}

\newcommand{\eps}{\varepsilon}

\usepackage{MnSymbol}

\usepackage{skak}

\usepackage{adjustbox}
\usetikzlibrary{arrows}
\usetikzlibrary{decorations.pathreplacing}

\title[ASEP via Mallows]
{ASEP via Mallows coloring}

\author{Alexei Borodin}

\address[Alexei Borodin]{Department of Mathematics, MIT, Cambridge, USA. E-mail: borodin@math.mit.edu}

\author{Alexey Bufetov}

\address[Alexey Bufetov]{Institute of Mathematics, Leipzig University, Germany. E-mail: alexey.bufetov@gmail.com}

\begin{document}

\begin{abstract}
	
In this paper we study the asymptotic behavior of the Asymmetric Simple Exclusion Process (=ASEP) with finitely many particles. It turns out that a certain randomized initial condition is the most amenable to such an analysis. Our main result is the behavior of such an ASEP in the KPZ limit regime. 
A key technical tool introduced in the paper --- the coloring of ASEP particles with the use of random Mallows permutations --- may be of independent interest. 

\end{abstract}

\maketitle

\section{Introduction}

The Asymmetric Simple Exclusion Process (or ASEP, for short) is a prototypical interacting particle system and a prominent member
of the KPZ (Kardar-Parisi-Zhang) universality class in (1+1) dimensions. 
It is a continuous-time Markov chain on the set of configurations of particles and holes on the integer lattice $\mathbb Z$. 
Particles move to the right with rate 1 and to the left with rate $q \in [0;1)$, provided that the target site is empty.
It is common to consider ASEP initial conditions with infinitely many particles, with the step initial condition, 
when particles occupy all nonnegative integers and only them, arguably being the simplest non-stationary one.

In this note we demonstrate an explicit connection between ASEPs with two different initial conditions. The first one 
is an arbitrary (possibly random) infinite subset $S$ of $\mathbb Z$ with the largest element existing almost surely.
The second one is a $K$-element subset $S_K$ of $S$ that is chosen according to a \emph{Mallows coloring} of $S$.
For $q=0$, $S_K$ is simply the $K$ right-most elements of $S$, which makes sense as in that case the evolution of 
those elements is independent of the rest of the system. 
For $q>0$, $S_K$ is the $K$ right-most elements of a permuted ordering of $S$, where the permutation is 
distributed according to the Mallows measure on infinite permutation (see Section \ref{sec:MalDef} below for its definition).

This connection turns out to be asymptotically useful, as we demonstrate in the following two examples.

First, we leave $q$ and $K$ fixed and let the time go to infinity in the case of the step initial condition. Using the known 
asymptotics of the infinite system from \cite{BO2}, we obtain an expression for the height function in the diffusive limit 
of the $K$-particle ASEP (this limit is actually independent of where those $K$ particles start from). The answer is given in 
terms of a multiplicative functional of a determinantal point process with the discrete Hermite kernel. For $q=0$, it is well known 
that this diffusive limit is described by the Warren process \cite{W} that is closely related to the GUE ensembles \cite{GS}, 
for which many formulas are available. For $q>0$, this kind of result appears to be new.  

Second, we consider the weakly asymmetric limit $q\to 1$, in which it is known that the ASEP converges to the KPZ equation, 
together with the same step initial condition. 
Our goal was to find a scaling of $K$ that would yield a limiting distribution that would be different from the one for the 
infinite system. This is indeed possible: At the very tail of our block of $K$ particles with appropriately tuned $K$, we observe that the properly scaled 
height function converges to a certain deterministic function of the usual narrow wedge KPZ limit. To our best knowledge, 
such a limiting object has not appeared in the ASEP context before.  

While we focus on one-point height distributions, the Mallows coloring technique can be used to study 
multi-point distributions as well; necessary ingredients are also provided in the text. 

Finally, let us note that while we prove results for the standard single-species ASEP, considering 
its colored or multi-species version seems essential for our proofs. 

\subsection*{Acknowledgments}

A. Borodin was partially supported by the NSF grants DMS-1664619,
DMS-1853981, and the Simons Investigator program. A. Bufetov was partially supported by the European Research Council (ERC), Grant Agreement No. 101041499. 

\section{The Mallows measure}

\subsection{Definition and properties}
\label{sec:MalDef}

We will use notations $\Z_n := \{1,2, \dots, n \}$ and $\N = \{1,2,3, \dots\}$. The goal of this section is to recall definitions and properties of \textit{Mallows random permutations} in the finite $\Z_n \to \Z_n$ and infinite $\N \to \N$ cases. Throughout the paper let $q$ be a real number, $0 \le q < 1$.

The Mallows measure $\mathcal{M}_n$ on the set of permutations $\Z_n \to \Z_n$ assigns a permutation $w$ the probability $q^{\mathrm{inv}(w)} Y_n$, where $\mathrm{inv}(w)$ is the number of inversions in $w$ and $Y_n$ is a normalisation constant given by
$$
Y_n := \prod_{j=1}^n \frac{1-q}{1-q^j}.
$$
Another description of the measure $\mathcal{M}_n$ can be given with the use of the so-called \textit{q-exchangeability} property. Let $w(i) = a$, $w(i+1) = b$, and let $w_{i,i+1}$ be a permutation that coincides with $w$ for all elements of $\Z_n$ with the exception that $w_{i,i+1} (i) = b$, $w_{i,i+1} (i+1) = a$. Then for any $i,a,b$ one has
$$
q \mathcal{M}_n (w) = \mathcal{M}_n (w_{i,i+1}) , \ \ \ \mbox{if $a<b$} ; \qquad \qquad  \mathcal{M}_n (w) = q \mathcal{M}_n (w_{i,i+1}) , \ \ \ \mbox{if $a>b$},
$$
and $\mathcal{M}_n$ is the unique probability measure on permutations that satisfies the property.

Consider now the case of infinite permutations $\N \to \N$. The first definition of the Mallows measure does not have a direct analog in this situation (each individual infinite permutation will have probability zero), however, as shown by Gnedin-Olshanski \cite{GO1}, the second definition can be generalized to this case in a very natural way. Supply the space of maps $\N \to \N$ with the sigma-algebra generated by cylinders, and let $\mathbf{w}$ be a cylinder in this set obtained by fixing the images of finitely many integers, including $\mathbf{w}(i) = a$, $\mathbf{w}(i+1) = b$. Let $\mathbf{w}_{i,i+1}$ be a cylinder that coincides with $\mathbf{w}$ for all fixed coordinates and their images with the exception that $\mathbf{w}_{i,i+1} (i) = b$, $\mathbf{w}_{i,i+1} (i+1) = a$. Then the \textit{Mallows measure} $\mathcal{M}_{\N}$ is a unique probability measure on the set of infinite permutations $\N \to \N$ which satisfies
$$
q \mathcal{M}_{\N} (\mathbf{w}) = \mathcal{M}_{\N} (\mathbf{w}_{i,i+1}) , \ \ \ \mbox{if $a<b$} ; \qquad \qquad  \mathcal{M}_{\N} (\mathbf{w}) = q \mathcal{M}_{\N} (\mathbf{w}_{i,i+1}) , \ \ \ \mbox{if $a>b$},
$$
for all $a$, $b$, and all cylinders.
See \cite{GO1} for a proof of the correctness of this definition.

The random permutation $\N \to \N$ distributed according to the measure $\mathcal{M}_{\N}$ can be sampled via the following algorithm. Let $G$ be a geometric distribution on $\N$ defined by
$$
G (i) = q^{i-1} (1-q), \qquad i=1,2,\dots.
$$
Let $\xi_1, \xi_2, \dots, \xi_{n}, \dots$ be an infinite sequence of independent random variables distributed according to $G$.
The algorithm below also uses a (dynamically changing) word which consists of integers from $\N$. Its initial form is $1 2 3 \dots n \dots$ (all integers are written in their linear order from left to right; each integer occurs exactly once), and this word will change during the algorithm.

At step 1 of the sampling algorithm we set $w(1)= \xi_1$. Then we remove the integer $w(1)$ from the word $1 2 3 4 \dots n \dots$. At step 2, we set $w(2)$ to be equal to the integer which stands at position $\xi_2$ (counting from the left) in the current word (at step 2 it can be $\xi_2+1$ or $\xi_2$, depending on whether $\xi_1 \le \xi_2$ or not). Then we remove the chosen letter from the word and iterate the procedure. At step number $k$, we set $w(k)$ to be equal to the integer which stands at position $\xi_k$ (counting from the left) in the current word, and then remove the chosen integer from the word. We iterate the process infinitely and obtain as a result a Mallows-random infinite permutation $w(1) w(2) \dots w(n) \dots$.


The described algorithm is due to Gnedin-Olshanski, see \cite[Section 4]{GO1} for a proof and a more detailed discussion.

\subsection{One-point height function}

Let $w$ be a random permutation distributed according to the Mallows measure $\mathcal{M}_{\N}$. Let us define a \textit{height function} $f_{\N,K,L}$, which is an integer-valued random variable with parameters $K,L \in \Z_{\ge 0}$, via
$$
f_{\N,K,L} := | \{ i : 1 \le i \le L, 1 \le w(i) \le K \}  |.
$$

The main goal of this section is to prove the following result.

\begin{proposition}
For any integers $s,K,L \ge 0$, $s \le \min(K,L)$, we have
\begin{equation}
\label{eq:MalHeightInf}
\mathrm{Prob} \left( f_{\N,K,L} =s \right) = q^{(K-s)(L-s)} \frac{\prod_{i=0}^{s-1} (1-q^{K-i}) (1-q^{L-i})}{\prod_{i=1}^s (1-q^i)}.
\end{equation}
\end{proposition}
\begin{proof}
For the proof we need to analyze the sampling algorithm from Section \ref{sec:MalDef}. We are interested in the prefix of the permutation $w(1) w(2) \dots w(L)$, and for each position we need to distinguish between only two variants: Whether the integer there is less or equal to $K$ or (strictly) greater than $K$. We will denote by $\alpha$ the integers from $1$ to $K$, and by $\beta$ the integers from $K+1$ to $\infty$. In these notations, the starting word involved in the algorithm is $\alpha \alpha \dots \alpha \beta \beta \dots \beta \dots$ with $K$ alphas and infinitely many betas. We need to run $L$ steps of the sampling algorithm, and to keep track of how the word evolves.

In more detail, we think about the prefix $w(1) w(2) \dots w(L)$ as a word from $\{ \alpha, \beta \}^L$, and this word is obtained by the procedure below. Set $\tilde K := K$; the value of $\tilde K$ may change during the algorithm. We iterate $L$ times the following steps for each of $i=1,2,\dots L$:

a) The letter at position $i$ is set to $\alpha$ with probability $\left( 1-q^{\tilde K} \right)$, and it is set to $\beta$ with probability $q^{\tilde K}$.

b) If $\alpha$ was chosen, then $\tilde K$ decreases by 1. If $\beta$ was chosen, then $\tilde K$ does not change.

\

It is readily verified that these steps correspond to the sampling algorithm above. In particular, $\tilde K$ tracks the number of the remaining integers which are $\le L$ in the word at each step of the sampling. This procedure assigns to each element of $\{ \alpha, \beta \}^L$ a probability, and we need to sum these probabilities over all words with exactly $s$ letters $\alpha$ in order to reach the claim of the proposition. 

Let us analyze one such sequence $w$. The $s$ letters $\alpha$ necessarily come with probability $(1-q^K) (1 -q^{K-1} ) \cdots (1-q^{K-s+1})$ (this expression does not depend on the sequence), while the letters $\beta$ have probability $ q^{K(L-s)-m(w)}$, where $m(w)$ is the number of inversions in $w$ (we assume that $\alpha < \beta$). Summing over all sequences $w$ and applying the combinatorial description of the $q$-Binomial coefficient, we obtain the expression
$$
(1-q^K) (1 -q^{K-1}) \cdots (1-q^{K-s+1}) q^{K(L-s)} \binom{L}{s}_{q^{-1}},
$$
which is identical to the formula from the statement. 

\end{proof}

\begin{remark}
Note that expression \eqref{eq:MalHeightInf} is symmetric in $K$ and $L$. While such a symmetry is not immediate from the definition or the sampling algorithm, it reflects the known fact that Mallows measures are invariant under the transform $w \mapsto w^{-1}$.
\end{remark}

\subsection{Multi-point height function}

Let $w$ be an infinite random permutation $\N \to \N$ distributed according to the Mallows measure $\mathcal{M}_{\N}$. Let us define a \textit{multi-point height function} as an $r$-dimensional vector of random variables of the following form:
\begin{multline}
\label{deff:MalHeightInfManyPoints}
\mathbf{f}_{\N, K} (L_1, [L_1+1;L_2], \dots, [L_{r-1}+1;L_r] ) \\ := \left( f_{\N,K,L_1}, f_{\N,K,L_2} - f_{\N,K,L_1}, f_{\N,K,L_3} - f_{\N,K,L_2}, \dots, f_{\N,K,L_r} - f_{\N,K,L_{r-1}} \right),
\end{multline}
for $K \ge 0$, $L_r \ge L_{r-1} \ge \dots \ge L_1$.
The main goal of this section is to prove the following result.

\begin{proposition}
For any $(s_1, s_2, \dots, s_r) \in \{0,1,2, \dots, \}^r$, such that $s_1+s_2+ \dots + s_r \le K$, we have
\begin{multline}
\label{eq:MalHeightInfManyPoints}
\mathrm{Prob} \left( \mathbf{f}_{\N, K} (L_1, [L_1+1;L_2], \dots, [L_{r-1}+1;L_r] ) = (s_1,s_2, \dots, s_r) \right) \\
= \mathrm{Prob} \left( f_{\N, K,L_1 } = s_1  \right) \mathrm{Prob} \left( f_{\N, K-s_1,L_2 - L_1 } = s_2 \right) \cdots \mathrm{Prob} \left( f_{\N, K-s_1-s_2-\dots-s_{r-1},L_r - L_{r-1} } = s_{r} \right) .
\end{multline}
\end{proposition}

\begin{proof}

For the proof we need to use the sampling algorithm from Section \ref{sec:MalDef} again. We start with the prefix of the permutation $w(1) w(2) \dots w(L_1)$. The chance that exactly $s_1$ of these letters are $\le K$ equals $\mathrm{Prob} \left( f_{\N, K,L_1} = s_1  \right)$ immediately from the definition of the algorithm. Therefore, $(K-s_1)$ slots remain free when we start sampling the places of integers from $[L_1+1;L_2]$. Since the random variables which are used in this sampling are independent from the ones used in the first step, we obtain that the chance that exactly $s_2$ integers from  
$[L_1+1;L_2]$ occupy these spots, is equal to $\mathrm{Prob} \left( f_{\N, K-s_1,L_{2} -L_1 } = s_1  \right)$. Iterating this argument, we arrive at the statement of the proposition.
\end{proof}

\begin{remark}
The right-hand side of \eqref{eq:MalHeightInfManyPoints} has an explicit product form due to \eqref{eq:MalHeightInf}.
\end{remark}

\section{Mallows coloring of ASEP particles}
\label{sec:mallows-coloring}

\subsection{Single-species ASEP}

Let $q \in [0;1)$. We consider ASEP on $\Z$. This is a continuous time Markov chain with a state space
\begin{equation*}
\Omega :=\left\{\xi\in \{0,1\}^{\Z} \right\}.
\end{equation*}
We think of the $1's$ as particles, and of the $0's$ as holes. The dynamics of ASEP can be described as follows: Each particle has two exponential clocks of rate 1 and $q$. When the clock of rate 1 rings, the particle attempts to make a unit step to the right. When the clock of rate $q$ rings, the particle attempts to make a unit step to the left. The attempt is successful if the target site is occupied by a hole: The hole and the particle exchange their positions when the particle moves a unit step.  If the attempt is not successful, nothing happens. All clocks of all particles are jointly independent. 

Let $S$ be a (possibly random) subset of $\Z$ such that the largest element of $S$ almost surely exists. Consider the ASEP evolution which starts from the initial configuration $S$; randomness involved in jumps of particles is assumed to be independent from the randomness of $S$. Let $S_t$ be the set of positions of particles after time $t$. Define the ASEP height function by
\begin{equation}
\label{eq:heightAsepDef}
h_{t,S} (x) := \left\{ \mbox{number of $y \in S_t$ such that $y \ge x$} \right\}, \qquad x \in \R.
\end{equation}

It follows from the standard graphical construction of ASEP (see \cite{H1,H2}, and also below) that under our assumptions, $h_{t,S} (x)$ is well-defined and almost surely finite for any $t$ and $x$.

\subsection{Multi-species ASEP}

While our main results deal with a single-species ASEP, it will be important for us to consider
ASEP with countably many colors  of particles; instead of colors we may speak interchangeably of types, classes, or species of particles.

The state space of this multi-species ASEP is the set of all maps $\Z \to \Z \cup \{ +\infty \}$ which we denote by $\mathfrak{S}$. For $w\in \mathfrak{S}$, we interpret an equality of the form $w(i)=j$ as the information that position $i$ is occupied by the particle with color $j$.
Our convention is that particles with lower color have priority over particles with higher color. Informally, this means that the particle with color $j$ treats all particles with color $>j$ as holes, while it is treated by all particles with color $<j$ as a hole.

Here we briefly recall the graphical construction of this multi-species ASEP which goes back to Harris \cite{H1} (see also, e.g., \cite{AAV}). Let $(\mathcal{P}(z),z\in\Z)$ be a collection of independent, rate $1$ Poisson processes  constructed on some probability space $(\hat{\Omega},\mathcal{A}, \mathbb{P} ).$ By $\mathcal{P}_{t}(z)$ we denote the value of the Poisson process on $\R_{\ge 0}$ at time $t$ (it counts the number of points in $[0;t]$ of the Poisson process $\mathcal{P} (z)$).
For fixed $t$, the independence of the Poisson processes implies that  for almost every $\omega\in \hat{\Omega}$ there is a sequence $(i_{n},n\in \Z)$ of integers such that
\begin{equation*}
\cdots <i_{n-2}< i_{n-1}<i_{n}<i_{n+1}<i_{n+2}<\cdots,
\end{equation*}
and $\mathcal{P}(i_{n})=0,n\in \Z.$ In words, this means none of Poisson processes $\{ \mathcal{P}_{t}(i_{n}) \}$ has a point in the interval $[0;t]$.

Given $w \in \mathfrak{S}$ and $z\in \Z$, we define a swapped configuration by
\begin{equation*}
\sigma_{z,z+1}(w)(i)=
\begin{cases}
w(i+1)  &\mathrm{for} \, i=z,\\
w(i-1)  &\mathrm{for} \, i=z+1,\\
w(i)   &\mathrm{else}.
\end{cases}
\end{equation*}
We can now describe the dynamics of the multi-species ASEP. 
When at time $\tau$ one of the Poisson processes $\mathcal{P}(z)$ has a jump, we update the ASEP configuration as follows: If $w_{\tau^{-}}(z)<w_{\tau^{-}}(z+1)$, then we update $w_{\tau}=\sigma_{z,z+1}(w_{\tau^{-}})$, whereas if
 $w_{\tau^{-}}(z)>w_{\tau^{-}}(z+1), $ we toss an independent coin such that with probability $1-q$, $w_{\tau}=w_{\tau^{-}},$ and with probability $q$ we have  $w_{\tau}=\sigma_{z,z+1}(w_{\tau^{-}})$. Note that to construct the process up to time $t,$ it suffices to apply these update rules inside each of the finite boxes $[i_{n}+1;i_{n+1}],$ and inside each box there are a.s. only finitely many jumps of the Poisson processes (in particular, there is a.s. a well-defined first jump) during the interval $[0,t]$. Also note that a.s. two jumps cannot happen at the same time, hence the graphical construction is well-defined.

The multi-species ASEP can be projected to various single-species ASEPs; in other words, one can view the multi-species ASEP as a coupling of several single-species ASEPs. Namely, let us fix an integer $k$ and identify all particles whose color lies in  $(-\infty,k]$ as particles (of a single-species ASEP), and all particles with color in $(k, +\infty)$ as holes, \emph{i.e.}, we map $w(\cdot)\mapsto 1_{(-\infty,k]}(w(\cdot)).$ This defines a map $X^{k}: \mathfrak{S}\to \{0,1\}^{\Z}$, the image of which is the ASEP on $\Z$ with particles and holes. For any $k \in \Z$ this projection produces a single-species ASEP; for varying $k$ these ASEPs are coupled.

\subsection{Mallows coloring}

Let $S$ be an arbitrary (possibly random) infinite subset of $\Z$ such that the largest element of it exists almost surely. Let $\mathrm{ord}: S \to \N$ be a unique strictly monotonically decreasing bijection; in words, we number all elements of $S$ by natural numbers from right to left. Let $\pi_{\N}$ be an infinite random permutation distributed according to the Mallows measure $\mathcal{M}_{\N}$. It is assumed to be independent of the randomness involved in $S$.  
We define $\mathcal{MC}_S$ as the superposition of the map $\pi_{\N}$ and the ordering map $\mathrm{ord}$:
$$
\mathcal{MC}_S (z) := \pi_{\N} \left( \mathrm{ord} (z) \right), \qquad z \in S,
$$
and say that $\mathcal{MC}_S: S \to \N$ is the $\textit{Mallows coloring}$ of the set $S$. Note that $\mathcal{MC}_S$ is a random map. 

Now let us consider a multi-species ASEP with the following initial configuration: Let $S$ be as above, place at all locations in $\Z \setminus S$ particles of type $+\infty$ (one can think about them as holes), and place at $z \in S$ a particle of color $\mathcal{MC}_S (z)$, for all $z \in S$. Denote by $S_t$ the set of positions which are occupied by particles with colors from $\N$ after running the ASEP evolution during time $t$. We have a map $\mathrm{Colors}_t: S_t \to \N$ which assigns to each position the color of the particle which is located there. The following proposition plays a key role in the present paper.

\begin{proposition}
\label{prop:MalColPreserve}
For any multi-species ASEP of the form described above, any $t \in \R_{\ge 0}$, and $\hat S := S_t$, the map $\mathrm{Colors}_t$ is the Mallows coloring of $\hat S$.
\end{proposition}
\begin{proof}
At time $t=0$ the map $\mathrm{Colors}_0$ satisfies the claim by our definition of the initial configuration.
For the time evolution of ASEP, we can split the basic updates into two disjoint groups. The first group consists of updates which affect 0 or 1 elements of $S_t$. These updates do not change the coloring inside $S_t$, though they might change the set $S_t$ itself. The second group consists of updates which affect 2 elements of $S_t$. The result of such an update might affect the coloring of $S_t$, but it does not change $S_t$. Due to the $q$-exchangeability property of the Mallows measure, the coloring always remains Mallows-distributed\footnote{ Indeed, assume that two neighboring positions contain particles of colors $i$ and $j$, for $i < j$. Due to the $q$-exchangeability, we have probability $\frac{1}{1+q}$ that they are in order $\dots j i \dots$, and probability $\frac{q}{1+q}$ that they are in order $\dots i j \dots$. Applying the update rules formulated above, we see that in the first case we get $\dots j i \dots$ with probability $(1-q)$ and $\dots i j \dots$ with probability $q$, while the second case produces $\dots j i \dots$ with probability 1. Collecting the terms, we obtain exactly the same probabilities, satisfying the $q$-exchangeability, after the update.}.
Moreover, since only updates from the second group change coloring, and only updates of the first group change $S_t$, the map $\mathrm{Colors}_t$ remains independent from randomness in $S_t$ for any $t \ge 0$.


\end{proof}

\begin{remark}
Note that the proof provides a stronger statement than formulated in the proposition, since it deals with the elementary swaps only. This allows to prove in a similar way the same statement for a different choice of time (for example, one can consider discrete time leading to the stochastic six vertex model), and also for a different choice of generator of a random walk on a Hecke algebra (see \cite{B}; for example, one can prove an analogous statement for q-TAZRP). We omit precise reformulations, since they can be obtained straightforwardly, and since we will not study asymptotics of other interacting particle systems in the present paper.
\end{remark}

\subsection{Distributional identities}
\label{sec:dist-id}

As before, let $S$ be an arbitrary (possibly random) infinite subset of $\Z$ such that its largest element exists almost surely.
Let $\mathcal{MC}_S (z)$ be its Mallows coloring, and define the set $\hat S_K$, $K \ge 0$, via
$$
\hat S_K := \{ \mathcal{MC}_S (i) \}_{1 \le i \le K}.
$$
For $q=0$ these are just the $K$ right-most elements of $S$. For $q \ne 0$, the set $\hat S_K$ is random (even if $S$ is deterministic), but still informally one can think of $\hat S_K$ as being ``not too far'' from the $K$ right-most positions.

Consider a single-species ASEP which starts from the initial configuration $S$, and recall that we denote its \emph{height function} \eqref{eq:heightAsepDef} by $h_{t,S}$. We will also consider a single-species ASEP which starts from the initial configuration $\hat S_K$, and its height function $h_{t,\hat S_K}$. We show that these two height functions are strongly related to each other.

\begin{theorem}
\label{th:One-Point-Coloring}
For any $x \in \R$ and $s,K,L \in \Z_{\ge 0}$, we have
\begin{equation}
\label{eq:One-Point-Coloring}
\mathrm{Prob} \left( h_{t,\hat S_K} (x) = s \right) = \sum_{L \ge s} \mathrm{Prob} \left( h_{t,S} (x) = L \right) \mathrm{Prob} \left( f_{\N,K,L} (x) =s \right).
\end{equation}
\end{theorem}
\begin{proof}
Consider the Mallows coloring of $S$ and the multi-species ASEP which starts from this coloring as the initial configuration (in the initial configuration positions $\Z \setminus S$ are assumed to be filled by holes of color $+\infty$). It is crucial that this multi-species process couples two single-species ASEPs involved in the statement of the theorem. Indeed, treating species of colors from $\N$ as particles and species of color $+\infty$ as holes, we recover the process which starts from $S$, while treating species of colors from $[1;K]$ as particles and species of color from $>K$ as holes, we see the process which starts from $\hat S_K$.

As before, let $S_t$ be the set of positions occupied by particles after time $t$. By Proposition \ref{prop:MalColPreserve}, the colors of particles inside the set are distributed according to the Mallows coloring. Assume that exactly $L$ elements of this set are greater or equal to $x$. Then we are interested in how many of these positions are occupied by particles of color $\le K$. This is exactly the quantity given by $f_{\N,K,L}$ (recall that \eqref{eq:MalHeightInf} provides an explicit formula for it). Summing over all possible values of $L$, we arrive at the statement of the theorem.
\end{proof}

\begin{theorem}
\label{th:Many-Point-Coloring}
For any $r \ge 1$, $x_1 \ge x_2 \ge \dots \ge x_r \in \R$, and $m_1 \le m_2 \le \dots \le m_r \in \Z_{\ge 0}$, we have
\begin{multline}
\mathrm{Prob} \left( h_{t,\hat S_K} (x_1) = m_1, h_{t,\hat S_K} (x_2) = m_2, \dots h_{t,S_K} (x_r) = m_r \right) \\ = \sum_{L_i: \forall i \ L_i \ge m_i, L_i \le L_{i+1}} \mathrm{Prob} \left( h_{t,S} (x_1) = L_1, h_{t,S} (x_2) = L_2, \dots h_{t,S} (x_r) = L_r \right) \\ \times \mathrm{Prob} \left( \mathbf{f}_{\N, K} (L_1, [L_1+1;L_2], \dots, [L_{r-1}+1;L_r] ) = (m_1,m_2-m_1, \dots, m_r-m_{r-1}) \right).
\end{multline}
\end{theorem}

\begin{proof}

We use the notations and the construction from the proof of Theorem \ref{th:One-Point-Coloring}. We are now interested in positions occupied by particles of a process that started from $\hat S_K$. By Proposition \ref{prop:MalColPreserve}, the joint distribution of these positions can be obtained by considering the Mallows coloring of the set $S_t$ and taking the distribution of particles with colors from 1 to $K$ in this coloring. Next, we consider the event that $L_1$ positions from $S_t$ are to the right of $x_1$, $\ldots$ , $L_r$ positions from $S_t$ are to the right of $x_r$. Then the probability that $m_1$ of the first $L_1$ positions are occupied by particles with colors from 1 to $K$, $m_2-m_1$ of the positions from $[ L_1+1; L_2 ]$ are also occupied by particles of colors from 1 to $K$, and so on, is given by 
$$
\mathrm{Prob} \left( \mathbf{f}_{\N, K} (L_1, [L_1+1;L_2], \dots, [L_{r-1}+1;L_r] ) = (m_1,m_2-m_1, \dots, m_r-m_{r-1}) \right),$$ 
which concludes the proof. 

\end{proof} 

\section{Asymptotics of ASEP with step initial condition}

\subsection{Discrete Hermite ensemble}

Let $H_n (x)$ be Hermite polynomials with (the most standard) normalization condition that their leading coefficient is equal to $2^n$. Define a function on pairs of nonnegative integers $(x,y)$ via
$$
K_{dH(r)} (x,y) := \left( \pi 2^{x+y} x! y! \right)^{-1/2} \int_{r}^{+\infty} H_x (t) H_y(t) \exp( -t^2) dt,
$$
where $r \in \R$ is a parameter. 

The determinantal point process with this kernel is called the \textit{discrete Hermite ensemble}. It was introduced in \cite{BO1}, see also \cite{BO2}. Let $\mathcal P_{dH(r)} = \left( p_1,p_2, \dots \right)$, $p_i \in \Z_{\ge 0}$, be the random point configuration of this ensemble. Define

\begin{equation*}
\mathcal{L}_{dH (r)}^{(q)} (z) := \mathbf{E} \prod_{p_i} \frac{1}{1+z q^{p_i}}, \qquad z \in \C \backslash \{ - q^{\Z_{\le 0}} \}, \qquad 0 \le q <1.
\end{equation*}

It will be useful for us to prove the following lemma.
\begin{lemma}
\label{lem:partic-case}
For any $r \in \R$ and $0 \le q<1$, we have
$$
\mathbf{E} \left( \sum_{p \in \mathcal P_{dH(r)} } q^p \right) = \frac{1}{\sqrt{\pi (1-q^2)}} \int_{r}^{+\infty} \exp \left( -t^2 \frac{1-q}{1+q} \right) dt.
$$
\end{lemma}

\begin{proof}
By the definition of correlation functions and determinantal point processes, we need to compute
the first $q$-moment of the first correlation function $\rho_1$:
\begin{equation}
\label{eq:HSS}
\sum_{n\ge 0} \rho_1 (n) q^n = \sum_{n \ge 0} \left( n! 2^n \sqrt{\pi} \right)^{-1} q^n \int_{r}^{+\infty} H_n (t)^2 \exp(-t^2) dt.
\end{equation}

A well-known formula for the generating series of Hermite polynomials (see, e.g., \cite[Ch. 10.13, (22)]{BE}) reads
\begin{equation}
\label{eq:HgenS}
\sum_{n=0}^{\infty} \frac{z^n}{2^n n!} H_n (x) H_n (y) = \frac{1}{\sqrt{1-z^2}} \exp \left( \frac{2xyz - (x^2+y^2) z^2}{1-z^2} \right),
\end{equation}
Plugging \eqref{eq:HgenS} into \eqref{eq:HSS} and swapping summation and integration (justification is straightforward), we obtain the statement of the lemma.
\end{proof}

\subsection{Single-species ASEP with step initial condition}

In this section we collect known results about ASEP with the step initial condition: This means that at the beginning of the process particles occupy all non-positive integers (and only them).  
We denote by $h_t (x)$ the height function of this process, defined as in \eqref{eq:heightAsepDef}. 

For a random variable $\xi$ which takes values in $\Z_{\ge 0}$, define its $q$-Laplace transform via
$$
L_{\xi}^{(q)} (z) := \mathbf{E} \left( \prod_{i \in \Z_{\ge 0}} \frac{1}{1+z q^{\xi+i}} \right) = \sum_{n \ge 0} \frac{(-z)^n \mathbf{E} \left( q^{n \xi} \right)}{(1-q)(1-q^2)\dots (1-q^n)}, \qquad z \in \C \backslash \{ - q^{\Z_{\le 0}} \}.
$$

The following proposition is \cite[Proposition 11.1]{BO2}. Note that its derivation in \cite{BO2} is based on certain connections with a theory of symmetric functions and does not use the technique of pioneering works of Tracy-Widom \cite{TW1}, \cite{TW2}.

\begin{proposition}
\label{prop:dH-BO}
For $q \in [0;1)$ as in the definition of ASEP, and any $r \in \R$, we have
\begin{equation*}
h_t \left( (1-q)t - r \sqrt{2(1-q) t} \right) \xrightarrow[t \to \infty]{d} \xi_r,
\end{equation*}
where $\xi_r$ is the random variable which takes values in $\Z_{\ge 0}$ ; its distribution is uniquely determined by the equality
\begin{equation}
\label{eq:dHmatch}
L_{\xi_r}^{(q)} (z) = \mathcal{L}_{dH (r)}^{(q)} (z), \qquad z \in \C \backslash \{ - q^{\Z_{\le 0}} \}.
\end{equation}
\end{proposition}


The following theorem is due to \cite{ACQ}, \cite{BG}, \cite{SS}, see also \cite[Theorem 11.6]{BO2}.

\begin{theorem}
\label{th:KPZ-step}
Let $q=1-\eps$, $t = \hat t \eps^{-4}$, $x= \hat x \eps^{-3}$, and assume that $\hat x / \hat t \in [0;1)$. Then 
$$
\hat \sigma \eps^{-2} - \ln (\eps) - \eps h_t (x) \xrightarrow[\eps \to 0]{} \hat \xi,  \qquad \mbox{with $\hat \sigma= \dfrac{(\hat t - \hat x)^2}{4 \hat t}$},
$$
where the convergence is in distribution, and $\hat \xi$ is a non-trivial random variable closely related to narrow wedge solution of the KPZ equation, see, e.g., \cite[Remark 11.7]{BO2} for more detail. 
\end{theorem}

\section{Asymptotics}

In this section we present the asymptotic applications of the Mallows coloring. In the first subsection, we address the asymptotics of the fixed amount of particles in ASEP (and also fixed $q$). In the second one, we study the KPZ-type limit, when the number of partciles goes to infinity and $q$ goes to 1. 

\subsection{Fixed amount of particles}

In the notations of Section \ref{sec:mallows-coloring}, let $S = \Z_{\le 0}$, and let $S_K$ be the random subset of $\Z_{\le 0}$ that is occupied by numbers $1,2, \dots, K$ from the Mallows coloring of $S$. In this section we study (a single-species) ASEP with $K$ particles which start at positions $S_K$. Recall that $h_{t,S_K} (x)$ denotes the number of particles in this ASEP which are weakly to the right of $x$ at time $t$. 

\begin{theorem}
\label{th:main-fin}
Let $K \in \Z_{\ge 1}$ be fixed. For any $r \in \R$ and $s \in \{0,1,\dots, K\}$, we have
\begin{multline*}
\lim_{t \to \infty} \mathrm{Prob} \left( h_{t,S_K} \left( (1-q)t + r \sqrt{2(1-q) t} \right) = s \right)
\\ = \sum_{L \ge s} \mathrm{Prob} \left( \xi_{-r} = L \right) q^{(K-s)(L-s)} \frac{\prod_{i=0}^{s-1} (1-q^{K-i}) (1-q^{L-i})}{\prod_{i=1}^s (1-q^i)},
\end{multline*}
where $\xi_{-r}$ is the uniquely determined by \eqref{eq:dHmatch} random variable, and the equality is in distribution.  
\end{theorem}

\begin{proof}
By Proposition \ref{prop:dH-BO} the height function $h_{t,S} \left( (1-q)t + r \sqrt{2(1-q) t} \right)$ for step initial condition converges in our scaling to the random variable $\xi_{-r}$. Applying Theorem \ref{th:One-Point-Coloring} and also \eqref{eq:MalHeightInf}, we arrive at the statement of the theorem. 

\end{proof}

\begin{corollary}
Let $P_1(t;K)$ be the position of the right-most particle of the $K$ particles ASEP started from $S_K$. We have
$$
\mathrm{Prob} \left( \frac{P_1(t;K) - (1-q)t}{\sqrt{2 (1-q) t} } \le r \right) = \mathbf{E} \left( q^{K \xi_{-r} } \right),
$$
where the expectation in the right-hand side is taken with respect to the distribution of the random variable $\xi_{-r}$.
\end{corollary}
\begin{proof}
One has
$$
\mathrm{Prob} \left( \frac{P_1(t;K) - (1-q)t}{\sqrt{2 (1-q) t} } \le r \right) = \mathrm{Prob} \left( h_{t,S_K} \left( (1-q)t + r \sqrt{2(1-q) t} \right) = 0 \right).
$$
After that, it remains to apply Theorem \ref{th:main-fin}. 
\end{proof}

\begin{remark}

For a consistency check, let us consider the (trivial) case $K=1$. By \eqref{eq:dHmatch} we have $\mathbf{E} q^{\xi_{-r}} = (1-q) \mathbf{E} \sum_{p \in \mathcal{P}_{dH(r)}} q^p$. Applying Lemma \ref{lem:partic-case}, we see that the limiting distribution of the only particle in the process after subtracting the drift and scaling by $\sqrt{t}$ is the centered Gaussian random variable with variance $(1+q)$, as expected.

\end{remark}

\begin{remark}
Theorem \ref{th:main-fin} provides an explicit formula for fluctuations of any particle in ASEP with fixed amount of particles for an arbitrary initial condition. Indeed, it is clear that any deterministic initial condition leads to the same fluctuations (this can be proven, e.g., via monotonicity). For fixed $q$, the randomness involved in the definition of our initial condition leads only to shifts of finite order, therefore ASEP started with it converges to the same distribution. 
\end{remark}

\begin{remark}
One can also provide a multi-point distributional identity between $h_{t,S_K}$ and $h_{t,S}$ with the use of Theorem \ref{th:Many-Point-Coloring} instead of Theorem \ref{th:One-Point-Coloring} in the proof of Theorem \ref{th:main-fin}.
\end{remark}

\subsection{KPZ type asymptotics}

We start with the analysis of the distribution $f_{\N,K,L} (s)$ in the KPZ type limit regime.

\begin{proposition}
\label{prop:KPZ-asymp}
Let $\hat \sigma \in \R$. Assume that there exist $c,d \in \R$, such that
$$
K=\hat \sigma \eps^{-3} -\ln( \eps) \eps^{-1} - c \eps^{-1} + o \left( \eps^{-1} \right) , \ \  L= \hat \sigma \eps^{-3} -\ln( \eps) \eps^{-1} - d \eps^{-1} + o \left( \eps^{-1} \right), 
\ \ q=1-\eps,
$$ 
and let $\tilde s (\eps)$ be the random variable chosen according to the distribution $f_{\N,K,L}$.

Then
$$
\left( L - \tilde s (\eps) \right) \eps \xrightarrow[\eps \to 0]{} \ln \left( 1 + e^{c-d} \right),
$$
where the convergence is in probability and is uniform over $d$ varying in a compact subset of a real line.
\end{proposition}

\begin{remark} Note that the claim implies that $\tilde s (\eps)$ satisfies a Law of Large Numbers in the limit $\eps \to 0$.
\end{remark}

\begin{proof}

Let us analyze the ratio of $f_{\N,K,L} (s)$ and $f_{\N,K,L} (s-1)$, for $s=1, \dots, \min(K,L)$: 

\begin{multline*}
\frac{f_{\N,K,L} (s)}{f_{\N,K,L} (s-1)} = \frac{\binom{K}{s}_q \binom{L}{s}_q q^{(K-s)(L-s)} (q;q)_s }{\binom{K}{s-1}_q \binom{L}{s-1}_q q^{(K-s-1)(L-s-1)} (q;q)_{s-1}} 
= \frac{(1-q^{K-s+1}) (1 - q^{L-s+1})}{1-q^s} q^{-K-L +2s -1}
\\ = \frac{1}{1 - q^{L- \hat s}} \left( 1 - q^{K-L+ \hat s +1} \right) \left( 1 - q^{\hat s +1} \right) q^{L-K - 2 \hat s -1}, 
\end{multline*}
where we set $\hat s = L -s$. Plugging in the expressions from the assumption, one obtains that the expression above is equal to 
\begin{equation}
\label{eq:ratio}
\frac{1}{1 - q^{L- \hat s}} (1- q^{\eps^{-1} ( d -c )} q^{\hat s +1}) (1 - q^{\hat s +1}) q^{\eps^{-1} ( c - d )} q^{-2 \hat s -1} + o(1). 
\end{equation}
The subsequent proof is essentially the following: If we assume for a moment that the variable $\hat s$ grows in such a way that $q^{\hat s}$ converges to $\mathbf{x} \in (0;1)$, then the expression above converges to 
\begin{equation}
\label{eq:ratio-parabola}
\frac{\left( 1- \exp( d -c ) \mathbf{x} \right) (1 - \mathbf{x}) \exp \left( c - d \right)} {\mathbf{x}^{2}}.
\end{equation}
By elementary calculus, one readily proves that this expression is equal to 1 for 
$$
\mathbf{x} = \frac{1}{1+\exp(c-d)},
$$ 
strictly greater than 1 for smaller values of $\mathbf{x}$, and strictly less than 1 for larger values of $\mathbf{x}$. Via a standard estimation of geometric series such a property of the ratio $\frac{f_{\N,K,L} (s)}{f_{\N,K,L} (s-1)}$ implies that $q^{L-\tilde s (\eps)}$ converges to $\left( 1 +\exp(c-d)\right)^{-1}$ in probability.

Let us now add details. Let $\delta>0$, and assume that $\hat s > \left( \ln \left( 1 + e^{c-d} \right) +\delta \right) \eps^{-1}$. Then 
$$
q^{\hat s}= (1-\eps)^{\hat s} < (1-\eps)^{\left( \ln \left( 1 + e^{c-d} \right) +\delta \right) \eps^{-1}}< \left( 1 + e^{c-d} \right)^{-1} \exp ( - \delta ).
$$
In this regime, the left-most factor from \eqref{eq:ratio}, namely $ \left( \frac{1}{1 - q^{L- \hat s}} \right)$, converges to 1. Due to the properties of the function \eqref{eq:ratio-parabola} mentioned above, we can conclude that there exists $C>1$ such that for all sufficiently small $\eps$ one has
\begin{equation}
\label{eq:ratio-ans1}
\frac{f_{\N,K,L} (s)}{f_{\N,K,L} (s-1)} > C,
\end{equation}
for such values of $\hat s = L-s$, and that the choice of $C$ is uniform over $d$ varying in compact subsets of $\R$. 

Assume now that $\hat s < \left( \ln \left( 1 + e^{c-d} \right) - \delta \right) \eps^{-1}$. Since the factor $ \left( \frac{1}{1 - q^{L- \hat s}} \right)$ is greater than one, we conclude that that the expression \eqref{eq:ratio} can be estimated asymptotically from below by the function \eqref{eq:ratio-parabola}. Therefore, there exists $D<1$ such that for all sufficiently small $\eps$ one has
\begin{equation}
\label{eq:ratio-ans2}
\frac{f_{\N,K,L} (s)}{f_{\N,K,L} (s-1)} <D,
\end{equation}
for such values of $\hat s = L-s$, and that the choice of $D$ is uniform over $d$ varying in compact subsets of $\R$. 

Combining \eqref{eq:ratio-ans1} and \eqref{eq:ratio-ans2}, we obtain the estimate via geometric series, which establishes the (uniform in $d$ in the sense above) convergence of $q^{L-\tilde s (\eps)}$ to $\left( 1 +\exp(c-d)\right)^{-1}$ in probability.
Since $q=1-\eps$, we obtain that $\left( L-\tilde s (\eps) \right) \eps$ converges to $\ln \left( 1 + e^{c-d} \right)$ in probability, and establish the claim of the proposition. 

\end{proof}

Let $S= \mathbb{Z}_{\le 0}$, and let $\hat S_K$ be the (random) set defined in Section \ref{sec:dist-id}. Consider a single-species ASEP with $K$ particles which starts with the initial condition $\hat S_K$. As before, let $h_{t,\hat S_K}$ be the height function of this process.

\begin{theorem}
\label{th:main-KPZ}
Let $c, \hat t \in \R$, let $K= \frac{\hat t}{4} \eps^{-3} -\ln( \eps) \eps^{-1} - c \eps^{-1}$, $q=1-\eps$, and $t=\eps^{-4} \hat t$. We have
$$
\frac{\hat t}{4 \eps^2} - \ln(\eps) - \eps h_{t,\hat S_K} (0) \to \alpha,
$$
where the convergence is in distribution, and $\alpha$ is a random variable defined by 
$$
\alpha := \hat \xi + \ln \left( 1 + \exp \left( c - \hat \xi \right) \right),
$$
where $\hat \xi$ is the random variable defined in the statement of Theorem \ref{th:KPZ-step}. 
\end{theorem}

\begin{proof}

By Theorem \ref{th:KPZ-step} 
$$
\lim_{\eps \to 0} \eps \left( h_{t,S} (0) - (\hat \sigma \eps^{-3} - \eps^{-1} \ln \eps) \right) = \hat \xi,
$$
where the convergence is in distribution. Equation \eqref{eq:One-Point-Coloring} implies that $h_{t,\hat S_K} (0)$ has the distribution $f_{\N,K,h_{t,S} (0)}$. Applying the Law of Large Numbers given by Proposition \ref{prop:KPZ-asymp} (and using uniformity in $d$ provided in the statement of that proposition), we obtain the convergence of $\eps \left( h_{t,S} (0) - h_{t,S_K} (0) \right)$ to $\ln \left( 1 + e^{c- \hat \xi} \right)$. This implies the statement of the theorem. 
\end{proof}

\begin{remark}
In the limit regime of Theorem \ref{th:main-KPZ}, we see different answers for the step initial condition and for the $K$ particles Mallows initial condition. Note, however, that in order to see the difference in the values of the height function at 0, we need to grow $K$ in a very particular way -- otherwise, new effects will not appear. 
\end{remark}

\begin{remark}
There exists a generalization of Theorem \ref{th:KPZ-step}, see \cite{BG}, \cite{ACQ}, for the multi-point convergence in space to a stationary stochastic process with the single-point distribution $\xi$. Using such a generalization and Theorem \ref{th:Many-Point-Coloring} instead of Theorem \ref{th:One-Point-Coloring}, one can extend Theorem \ref{th:main-KPZ} to a multi-point convergence for $K$ particles. We leave this to future work. 
\end{remark}

\begin{remark}
In the $q = 1 - \eps \to 1$ regime, the Mallows distribution grows away from the identity permutation (in comparison with the case of fixed $q$). Thus, our Mallows initial condition is significantly different from the configuration in which $K$ particles are packed next to each other. In particular, it easily follows from the known properties of the Mallows measure (see, e.g., \cite[Section 1]{BP}) that there will be of order $\eps^{-1} = \frac{1}{1-q}$ gaps between particles in our initial condition. Since our fluctuations also live on the scale $\eps^{-1}$, this limiting regime is sensitive to such a change of the initial condition.  
\end{remark}

\end{document}